\documentclass[11pt, reqno, twoside, makeidx]{amsart}
\usepackage[margin=2.7cm, marginpar=1cm]{geometry}

\usepackage{wrapfig}
\usepackage{marginnote}
\usepackage{hyperref}
\usepackage[english]{babel}
\usepackage[utf8]{inputenc}
\usepackage[T1]{fontenc}
\usepackage{amsmath,amsthm,amssymb,amsfonts}
\numberwithin{equation}{section}
\usepackage{enumerate}
\usepackage{enumitem}
\usepackage{faktor}
\usepackage{dsfont}
\usepackage{nicematrix}
\usepackage{scalerel,stackengine}
\usepackage{textcomp}
\usepackage{graphicx}
\usepackage{extarrows}
\usepackage{epigraph}
\usepackage{graphicx}
\usepackage{subcaption}
\usepackage{sidecap}
\usepackage{indentfirst}
\usepackage{tikz}
\usepackage{tikz-cd}
\usepackage{tkz-euclide}
\usetikzlibrary{shapes.misc,arrows,matrix,patterns,decorations.markings,positioning}
\tikzset{cross/.style={cross out, draw=black, minimum size=2*(#1-\pgflinewidth), inner sep=0pt, outer sep=0pt},
cross/.default={4.5pt}}
\usepackage{pgfplots}
\usepackage{caption}
\usepackage{float}
\usepackage{color}
\usepackage{epstopdf}
\usepackage{epsfig}
\usepackage{stmaryrd}
\usepackage{color}
\usepackage{geometry}
\usepackage{multicol}
\usepackage{verbatim}
\usepackage{wrapfig}
\usepackage{float}

\renewcommand{\geq}{\geqslant}
\renewcommand{\leq}{\leqslant} 
\renewcommand{\epsilon}{\varepsilon}

\newcommand{\Z}{\mathbb{Z}}
\newcommand{\Q}{\mathbb{Q}}

\DeclareFontFamily{U}{mathx}{\hyphenchar\font45}
\DeclareFontShape{U}{mathx}{m}{n}{
      <5> <6> <7> <8> <9> <10>
      <10.95> <12> <14.4> <17.28> <20.74> <24.88>
      mathx10
      }{}
\DeclareSymbolFont{mathx}{U}{mathx}{m}{n}
\DeclareFontSubstitution{U}{mathx}{m}{n}
\DeclareMathAccent{\widecheck}{0}{mathx}{"71}
\DeclareMathAccent{\wideparen}{0}{mathx}{"75}

\newtheorem{teo}{Theorem}[section]
\newtheorem*{teo*}{Theorem}
\newtheorem{lemma}[teo]{Lemma}
\newtheorem{prop}[teo]{Proposition}
\newtheorem*{prop*}{Proposition}

\usepackage{xpatch}
\makeatletter
\xpatchcmd{\@thm}{\thm@headpunct{.}}{\thm@headpunct{}}{}{}
\makeatother

\pgfplotsset{compat=1.18}
\begin{document}
\title[Brieskorn spheres with two fillable contact structures]{Brieskorn spheres with two fillable contact structures}
\author{Alberto Cavallo}
\address{HUN-REN Alfr\'ed R\'enyi Insitute of Mathematics, Budapest 1053, Hungary}
\email{acavallo@impan.pl}

\begin{abstract}
 Applying our recent classification of negative-twisting tight contact structures on Seifert fibered spaces whose base orbifold is a sphere, we provide the complete list of all the Brieskorn spheres carrying at most two symplectically fillable structures, up to isotopy, compatible with a given orientation. 
\end{abstract}

\maketitle

\thispagestyle{empty}

\section{Introduction}
Brieskorn spheres are prominent in low-dimensional topology because they provide a rich and tractable family of examples where the interaction of singularity theory with 3-manifolds, and smooth 4-manifolds can be studied explicitly. Arising as links of isolated complex surface singularities, see \cite{Saveliev} for more details, they admit a description as Seifert fibered spaces and as boundaries of plumbed 4-manifolds, making them especially useful for concrete computations in gauge theory and Floer theory \cite{FS,OSz-negative,OSz-fullpath}.

From singularity theory we know that Brieskorn spheres always carry a Stein fillable structure, with at least one orientation; namely, the Milnor fillable structure $\xi_\text{can}$ obtained by restricting the field of complex tangent planes to the link of the corresponding singularity. Moreover, from the results in \cite{LM} a (Stein) fillable structure actually exists on both orientations, except for $-\Sigma(2,3,5)$ which has no tight structure at all, see \cite{EH}.

Very recently, together with Matkovi\v c we classified \cite{CM-negative,CM} all the negative-twisting tight contact structures on the Seifert fibered spaces $M(e_0;r_1,...,r_n)$, and Brieskorn spheres are among these manifolds. It follows from our results, see \cite[Theorem 1.9]{CM}, that $\Sigma(2,3,5),$ $-\Sigma(2,3,7)$ and $-\Sigma(2,3,11)$ are the only Brieskorn spheres which carry a unique tight structure, up to isotopy. 

Later on, in a joint work with Alfieri \cite{ACM} we also determined which canonically oriented Brieskorn spheres have vanishing Heegaard Floer correction term and admit at most two tight structures. Note that it is known from the literature \cite[Corollary 9.7 and Theorem 1.11]{OSz-negative} that, for a Brieskorn sphere $Y$, the correction term is zero if and only if the intersection form of its plumbing star-shaped graph is diagonalizable; moreover, by \cite[Corollary 3]{ImC-e} this can happen only when, after orienting $Y$ so that its Seifert coefficients satisfy $e(Y):=r_1+...+r_n+e_0<0$, one has $e_0=-1$.

Here, we build upon these results and complete the list of all the Brieskorn spheres, with a given orientation, carrying at most two symplectically fillable structures. We recall that a Seifert fibered space is called \emph{small} if its standard graph has three legs.

\begin{teo}
 \label{teo:two}
 The only Brieskorn spheres which carry two fillable structures, up to isotopy, compatible with a given orientation are listed in Table \ref{Table}, together with their invariants. Among these, the small Brieskorn spheres carry no other tight structure, while the others also carry non-fillable tight structures which contain Giroux torsion.
\end{teo}

Note that when $Y$ has exactly two fillable structures then they are the Milnor fillable one $\xi_\text{can}$, and its conjugate $\overline\xi_\text{can}$. Therefore, on each $Y$ small in Table \ref{Table} there is always a unique tight structure up to contactomorphism. On a generic Seifert fibered space the latter one is a much weaker condition; nonetheless, on Brieskorn spheres Theorem \ref{teo:two} implies the following result.

\begin{table}[ht]
 \centering\renewcommand{\arraystretch}{2}
 \begin{NiceTabular}{|c|c|c||c|c|c|}
  \hline 
  $Y=\Sigma(a_1,a_2,a_3)$ & $d_3(\xi_\text{can})$ & $d(Y)$ & $Y=\Sigma(a_1,a_2,a_3)$ & $d_3(\xi_\text{can})$ & $d(Y)$ \\
  \hline\hline 
  $\Sigma(2,3,6k+1)$ with $k\geq1$ & 0 & 0 & $\Sigma(2,3,6k-1)$ with $k\geq2$ & 2 & 2 \\
  \hline  $\Sigma(2,5,7)$ & 0 & 0 & $\Sigma(2,5,9)$ & 2 & 2 \\
  \hline  $\Sigma(3,4,5)$ & 0 & 0 & $\Sigma(3,7,10)$ & $-6$ & 2 \\ 
  \hline  $\Sigma(2,7,11)$ & 0 & 2 & $\Sigma(3,7,19)$ & $-18$ & 2 \\ 
  \hline  $\Sigma(3,4,11)$ & 0 & 2 & $\Sigma(3,8,11)$ & $-10$ & 2 \\ 
  \hline  $\Sigma(3,5,7)$ & 0 & 2 & $\Sigma(4,5,9)$ & $-6$ & 2 \\ 
  \hline   &  &  & $\Sigma(3,5,14)$ & $-4$ & 4 \\ 
  \hline\hline
  $Y=\Sigma(a_1,a_2,a_3,a_4)$ & $d_3(\xi_\text{can})$ & $d(Y)$ & $Y=\Sigma(a_1,a_2,a_3,a_4)$ & $d_3(\xi_\text{can})$ & $d(Y)$ \\
  \hline
  $\Sigma(2,3,7,41)$ & $-418$ & 12 & $\Sigma(2,3,11,13)$ & $-208$ & 6 \\
  \hline
 \end{NiceTabular}
 \caption{\smaller[1]{The complete list of all the canonically oriented Brieskorn spheres whose only fillable structures are $\xi_\text{can}$ and $\overline\xi_\text{can}$, and their $d_3$-invariant and correction term $d$.}}  
 \label{Table}\renewcommand{\arraystretch}{1}
\end{table}

\begin{prop}
 \label{prop:contactomorphism}
 The only Brieskorn spheres carrying a unique fillable structure up to contactomorphism are $\Sigma(2,3,5),-\Sigma(2,3,7)$ and $-\Sigma(2,3,11)$, together with the ones listed in Table \ref{Table}.   
\end{prop}

For completeness we mention that the correction terms of $\Sigma(2,3,5),-\Sigma(2,3,7)$ and $-\Sigma(2,3,11)$ are equal to $2,0$ and $-2$ respectively.

\subsection*{Acknowledgments} {\smaller[1] I would like to thank the Matematiska institutionen at Uppsala universitet for their friendly hospitality. The author has been partially supported by the HORIZON-ERC-2023-ADG 101141468 KnotSurf4d project. }

\section{Brieskorn spheres with at most two structures}
We determine when $Y=\Sigma(a_1,...,a_n)$ carries exactly two fillable structures. We recall that the manifold $Y$ always has at least one Stein fillable structure; namely, the structure $\xi_\text{can}$. 

\subsection{The case when \texorpdfstring{$\mathbf{d_3}$}{d3} is vanishing}
\label{subsection:first}
We prove that the only canonically oriented Brieskorn spheres $Y=\Sigma(a_1,...,a_n)$ as above and with vanishing $d_3(\xi_{\text{can}})$ are $\Sigma(3,4,5),$ $\Sigma(2,5,7)$ and the family $\Sigma(2,3,6k+1)$ for $k\geq1$, which bound smooth rational homology balls \cite{Savk} thus they also have vanishing correction term, and $\Sigma(2,7,11),$ $\Sigma(3,5,7)$ and $\Sigma(3,4,11)$, see Figures \ref{Family3} and \ref{Family1}.

We recall our terminology in \cite{CM-negative,CM,ACM}, see also \cite{Saveliev}: given a Seifert fibered space $M=M(e_0;r_1,...,r_n)$ where $e_0\in\Z$ and $r_i\in\Q\cap(0,1)$ we define its standard graph $\Gamma$ as the star-shaped graph whose center as framing $e_0$, while each leg has vertices given by the negative continued fraction $-\frac{1}{r_i}=[m^i_1,...,m^i_{k_i}]$; in particular, on each leg the framings can be at most $-2$. In addition, the plumbed 4-manifold determined by $\Gamma$ has negative-definite intersection form if and only if $e(M)<0$; this is the case of canonically oriented Brieskorn spheres.

\begin{figure}[t] 
     \begin{tikzpicture}[scale=0.7]
    \tkzDefPoints{0/0/A, 1.5/1/B, 1.5/-1/D, 3/-1/E, 1.5/0/C}  
    \tkzDrawSegment(A,B)\tkzDrawSegment(A,D)\tkzDrawSegment(E,D)\tkzDrawSegment(B,A)\tkzDrawSegment(A,C)
    \tkzDrawPoints[fill,black,size=5](A,B,C,D,E)
     \tkzLabelPoint[above left](A){$-1$} 
     \tkzLabelPoint[above left](B){$-3$}\tkzLabelPoint[right](C){$-4$}
     \tkzLabelPoint[below left](D){$-3$}\tkzLabelPoint[below left](E){$-2$}
       \end{tikzpicture}\hspace{1cm} 
       \begin{tikzpicture}[scale=0.7]
    \tkzDefPoints{0/0/A, 1.5/1/B, 1.5/-1/D, 3/-1/E, 1.5/0/C}  
    \tkzDrawSegment(A,B)\tkzDrawSegment(A,D)\tkzDrawSegment(E,D)\tkzDrawSegment(B,A)\tkzDrawSegment(A,C)
    \tkzDrawPoints[fill,black,size=5](A,B,C,D,E)
     \tkzLabelPoint[above left](A){$-1$} 
     \tkzLabelPoint[above left](B){$-2$}\tkzLabelPoint[right](C){$-5$}\tkzLabelPoint[below left](D){$-4$}\tkzLabelPoint[below left](E){$-2$}
       \end{tikzpicture}\hspace{1cm}
     \begin{tikzpicture}[scale=0.7]
    \tkzDefPoints{0/0/A, 1.5/1/B, 1.5/-1/D, 3/-1/E, 5/-1/F, 1.5/0/C}  
    \tkzDefPoints{3.5/-1/X, 4.5/-1/Y}\tkzDefPoints{3.9/-1/P, 4/-1/Q, 4.1/-1/R}
    \tkzDrawPoints[fill,black,size=1](P,Q,R)
    \tkzDrawSegment(A,B)\tkzDrawSegment(A,D)\tkzDrawSegment(E,D)\tkzDrawSegment(E,X)\tkzDrawSegment(Y,F)\tkzDrawSegment(A,C)
    \tkzDrawPoints[fill,black,size=5](A,B,C,D,E,F)
     \tkzLabelPoint[above left](A){$-1$} \tkzLabelPoint[above left](B){$-2$}\tkzLabelPoint[right](C){$-3$}
     \tkzLabelPoint[below left](D){$-7$}\tkzLabelPoint[below left](E){$-2$}\tkzLabelPoint[below left](F){$-2$}
\end{tikzpicture}
     \caption{\smaller[1]{The standard graph of $\Sigma(3,4,5)$ (left), $\Sigma(2,5,7)$ (middle) and $\Sigma(2,3,6k+1)$ with $k\geq1$ (right). There are $k-1$ vertices with framing $-2$ on the third leg of the graph of $\Sigma(2,3,6k+1)$.}}
     \label{Family3}
\end{figure}
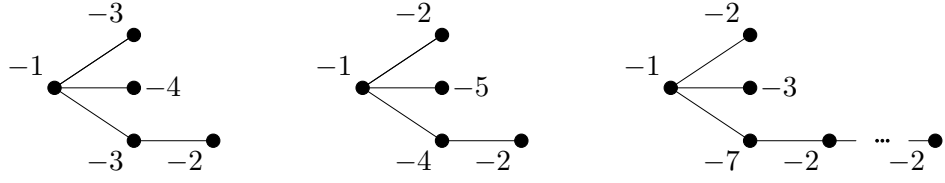 

We already know from \cite[Proposition 2.2]{ACM} that if a $Y$ with the properties we want has standard graph which satisfies $e_0=-1$ then it is as in Figure \ref{Family3}; hence, here we assume $e_0\leq-2$. In addition, we proved in \cite[Theorem 1.1]{CM-negative} that all the fillable structures on a negative-definite Seifert fibered space are presented by Legendrian surgery on the link obtained by blowing down $\Gamma$; in the case we have here, each framing is at most $-2$; for this reason, there is no blow-down available and the link is then given by $\Gamma$ itself. The number of structures is determined combinatorially by the Seifert coefficients as explained in \cite[Section 5]{CM-negative} and it is equal to \[|e_0+1|\cdot\prod_{i=1}^n\prod_{j=1}^{k_i}|m_{j}^i+1|\:;\] therefore, we can have two fillable structures only when there is a unique vertex in $\Gamma$ with framing $-3$, while all the others have framing $-2$. 

We start by assuming that $e_0=-3$: in this case the number of legs is at least $4$; otherwise, the graph would have no bad vertices and $Y$ would be an $L$-space \cite{OSz-fullpath}, implying that $Y$ is either $S^3$ or the Poincar\'e sphere $\Sigma(2,3,5)$. The latter manifold has $\Gamma=-E_8$ whose center is not $-3$. If there were at least 6 legs then \[e(Y)=e_0+r_1+...+r_n\geq-3+6\cdot\dfrac{1}{2}=0\:,\] thus $Y$ cannot be negative-definite. 

We consider the case when there are $n=4$ legs. We take the dual of the standard graph $\Gamma$, which represents $-Y=M(-1;\frac{1}{a},\frac{1}{b},\frac{1}{c},\frac{1}{d})$ with $a,b,c,d\geq2$ integers. Since we want $Y$ to be a Brieskorn sphere, we obtain that $a,b,c$ and $d$ are also pairwise coprime and $|H_1(Y)|=1$; hence, they need to satisfy the condition \[\dfrac{1}{a}+\dfrac{1}{b}+\dfrac{1}{c}+\dfrac{1}{d}=1+\dfrac{1}{abcd}\] which admits solutions, as we show in Proposition \ref{prop:number_theory2}, but since the vector $V_{\text{can}}$ associated to $\Gamma$ has all the coordinates vanishing, except a single one which is $-1$, we immediately obtain\footnote{\label{mat}The linear algebra computations were performed using the online tools available at \url{https://matrixcalc.org/}\:.} that \begin{equation}d_3(\xi_{\text{can}})=\dfrac{V_{\text{can}}^TQ^{-1}V_{\text{can}}+|\Gamma|}{4}=\dfrac{1}{4}\Big(-abcd+a+b+c+d-3\Big)
      \label{eq:d33}
     \end{equation} which clearly is never zero as \[abcd\geq(a+b)(c+d)>a+b+c+d\:.\]

We now assume that there are $n=5$ legs. Similarly to the previous case, we have that $-Y=M(-2;\frac{1}{a},\frac{1}{b},\frac{1}{c},\frac{1}{d},\frac{1}{e})$ with $a,b,c,d,e\geq2$ pairwise coprime integers. Then \[0<e(-Y)=\dfrac{1}{a}+\dfrac{1}{b}+\dfrac{1}{c}+\dfrac{1}{d}+\dfrac{1}{e}-2\leq\dfrac{1}{2}+\dfrac{1}{3}+\dfrac{1}{5}+\dfrac{1}{7}+\dfrac{1}{11}-2<0\]
leading to no possible solution, where here we are assuming that $e(-Y)>0$ as we want $Y$ to be canonically oriented.
     
     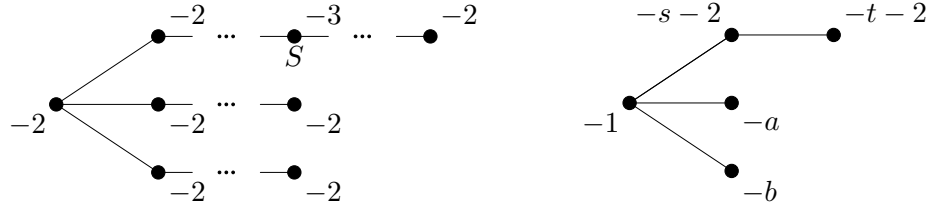
\begin{figure}[ht] 
     \begin{tikzpicture}[scale=0.9]
    \tkzDefPoints{0/0/A, 1.5/1/B, 1.5/-1/D, 1.5/0/C, 3.5/1/E, 5.5/1/F, 3.5/0/G, 3.5/-1/H}  
     \tkzDefPoints{2/1/X, 3/1/Y}\tkzDefPoints{2.4/1/P, 2.5/1/Q, 2.6/1/R}\tkzDefPoints{4/1/X', 5/1/Y'}\tkzDefPoints{4.4/1/P', 4.5/1/Q', 4.6/1/R'}
     \tkzDefPoints{2/0/X'', 3/0/Y''}\tkzDefPoints{2.4/0/P'', 2.5/0/Q'', 2.6/0/R''}
     \tkzDefPoints{2/-1/X''', 3/-1/Y'''}\tkzDefPoints{2.4/-1/P''', 2.5/-1/Q''', 2.6/-1/R'''}
     
    \tkzDrawSegment(A,B)\tkzDrawSegment(A,D)\tkzDrawSegment(A,C)\tkzDrawSegment(B,X)\tkzDrawSegment(E,Y)\tkzDrawSegment(E,X')\tkzDrawSegment(F,Y')
    \tkzDrawSegment(C,X'')\tkzDrawSegment(G,Y'')\tkzDrawSegment(D,X''')\tkzDrawSegment(H,Y''')
    \tkzDrawPoints[fill,black,size=5](A,B,C,D,E,F,G,H)\tkzDrawPoints[fill,black,size=1](P,Q,R,P',Q',R',P'',Q'',R'',P''',Q''',R''')
    \tkzLabelPoint[below left](A){$-2$}\tkzLabelPoint[above right](B){$-2$}\tkzLabelPoint[below right](C){$-2$}\tkzLabelPoint[below right](D){$-2$}\tkzLabelPoint[above right](E){$-3$}\tkzLabelPoint[above right](F){$-2$}\tkzLabelPoint[below right](G){$-2$}\tkzLabelPoint[below right](H){$-2$}
    \tkzLabelPoint[below](E){$S$}
       \end{tikzpicture}\hspace{1cm}
       \begin{tikzpicture}[scale=0.9]
    \tkzDefPoints{0/0/A, 1.5/1/B, 1.5/-1/D, 1.5/0/C, 3/1/E}  
    \tkzDrawSegment(A,B)\tkzDrawSegment(A,D)\tkzDrawSegment(B,A)\tkzDrawSegment(A,C)\tkzDrawSegment(B,E)
    \tkzDrawPoints[fill,black,size=5](A,B,C,D,E)
     \tkzLabelPoint[below left](A){$-1$} 
     \tkzLabelPoint[above left](B){$-s-2$}\tkzLabelPoint[below right](C){$-a$}\tkzLabelPoint[below right](D){$-b$}\tkzLabelPoint[above right](E){$-t-2$}
       \end{tikzpicture}
     \caption{\smaller[1]{The standard graph $\Gamma$ representing $Y$ (left) and its dual graph (right) representing $-Y$, when $e_0=-2$. There are $s$ vertices with framing $-2$ on the first leg before $S$, and $t$ ones after $S$; moreover, the number of vertices with framing $-2$ in $\Gamma$ is $a-1$ on the second leg, and $b-1$ on the third one.}}
     \label{Family2}
\end{figure}

We now assume that $e_0=-2$, and that the vertex $S$ with framing $-3$ appears on the first leg; the same argument of the previous case allows us to restrict to $n=3,4$. When $n=4$ we have that the first leg starts with the $-3$-vertex, and then continues with a string of $t$ vertices with framing $-2$; hence, one has $-Y=M(-2;\frac{t+2}{2t+3},\frac{1}{a},\frac{1}{b},\frac{1}{c})$ with $a,b,c\geq2$ integers. Then we need \[\dfrac{1}{a}+\dfrac{1}{b}+\dfrac{1}{c}+\dfrac{t+2}{2t+3}=2+\dfrac{1}{abc(2t+3)}\] which has no solutions, see Proposition \ref{prop:new}.

When $n=3$ the first leg contains a string of $s$ (resp. $t$) vertices with framing $-2$ before (resp. after) the vertex $S$.  From a direct computation one has $-Y=M(-1;\frac{t+2}{(t+2)(s+2)-1},\frac{1}{a},\frac{1}{b})$ with $s,t\geq0$ and $a,b\geq2$ integers, see Figure \ref{Family2}.

Using the Schur complement as before\footref{mat}, we have that \begin{equation}d_3(\xi_{\text{can}})=\dfrac{1}{4}\Big(s+t+a+b-(t+1)[ab-(s+1)(ab-a-b)]\Big)\:;
      \label{eq:d3}
     \end{equation} 
moreover, from the dual graph of $\Gamma$ we can determine the order of the first homology group of $M$: \[|H_1(Y)|= (t+2)ab-[(s+2)(t+2)-1](ab-a-b)\:.\] Hence, together with the fact that $a,b$ and $(s+2)(t+2)-1$ need to be pairwise coprime, we also impose that $|H_1(Y)|=1$; moreover, by assumption the identity $d_3(\xi_{\text{can}})=0$ should also hold.
     
We find only three Brieskorn spheres of this kind, see Figure \ref{Family1}; the proof involves some elementary number theory, and we postpone it to Proposition \ref{prop:number_theory}. Using the full path algorithm in \cite{OSz-fullpath}, we find the characteristic vectors \[V_1=\setcounter{MaxMatrixCols}{30}\begin{pmatrix} 0 & 0 & -1 & 0 & 0 & 0 & 0 & 0 & 0 & 0 & 0 & 2\end{pmatrix}\in H^2(X_{\Sigma(2,7,11)};\Z)\:,\] \[V_2=\setcounter{MaxMatrixCols}{30}\begin{pmatrix} 0 & -1 & 0 & 0 & 0 & 0 & 0 & 0 & 0 & 0 & 0 & 2\end{pmatrix}\in H^2(X_{\Sigma(3,5,7)};\Z)\] and \[V_3=\setcounter{MaxMatrixCols}{30}\begin{pmatrix} 0 & -1 & 0 & 0 & 0 & 0 & 0 & 0 & 0 & 0 & 0 & 0 & 0 & 0 & 2\end{pmatrix}\in H^2(X_{\Sigma(3,4,11)};\Z)\:,\] where $X$ is the negative-definite Stein domain given by $\Gamma$. The full path of the vectors $V_1, V_2$ and $V_3$ ends correctly, and it can be checked computationally that it realizes the maximal Maslov grading; therefore, by \cite[Theorem 1.2]{OSz-fullpath} one obtain that the corresponding correction terms are equal to $M(V_1)=M(V_2)=M(V_3)=2$. 

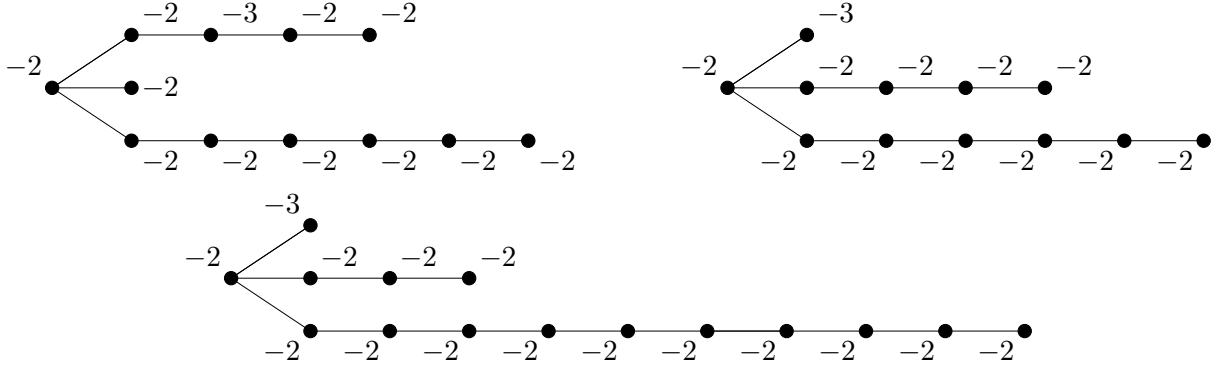
\begin{figure}[t] 
     \begin{tikzpicture}[scale=0.7]
    \tkzDefPoints{0/0/A, 1.5/1/B, 1.5/-1/D, 3/-1/E, 4.5/-1/F, 1.5/0/C, 3/1/H, 4.5/1/I, 6/1/G}  
    \tkzDefPoints{6/-1/J, 7.5/-1/K, 9/-1/L} 
    \tkzDrawSegment(A,B)\tkzDrawSegment(A,D)\tkzDrawSegment(E,D)\tkzDrawSegment(B,A)\tkzDrawSegment(A,C)\tkzDrawSegment(B,H)\tkzDrawSegment(I,H)
    \tkzDrawSegment(I,G)\tkzDrawSegment(E,F)\tkzDrawSegment(J,F)\tkzDrawSegment(J,K)\tkzDrawSegment(K,L)
    \tkzDrawPoints[fill,black,size=5](A,B,C,D,E,F,G,H,I,J,K,L)
     \tkzLabelPoint[above left](A){$-2$} 
     \tkzLabelPoint[above right](B){$-2$}\tkzLabelPoint[above right](H){$-3$}\tkzLabelPoint[above right](I){$-2$}\tkzLabelPoint[right](C){$-2$}\tkzLabelPoint[above right](G){$-2$}
     \tkzLabelPoint[below right](D){$-2$}\tkzLabelPoint[below right](E){$-2$}\tkzLabelPoint[below right](F){$-2$}\tkzLabelPoint[below right](J){$-2$}\tkzLabelPoint[below right](K){$-2$}\tkzLabelPoint[below right](L){$-2$}
       \end{tikzpicture}\hspace{1cm} 
       \begin{tikzpicture}[scale=0.7]
    \tkzDefPoints{0/0/A, 1.5/1/B, 1.5/-1/D, 3/-1/E, 4.5/-1/F, 1.5/0/C, 3/0/H, 4.5/0/I, 6/0/G}  
    \tkzDefPoints{6/-1/J, 7.5/-1/K, 9/-1/L} 
    \tkzDrawSegment(A,B)\tkzDrawSegment(A,D)\tkzDrawSegment(E,D)\tkzDrawSegment(B,A)\tkzDrawSegment(A,C)\tkzDrawSegment(C,H)\tkzDrawSegment(I,H)
    \tkzDrawSegment(I,G)\tkzDrawSegment(E,F)\tkzDrawSegment(J,F)\tkzDrawSegment(J,K)\tkzDrawSegment(K,L)
    \tkzDrawPoints[fill,black,size=5](A,B,C,D,E,F,G,H,I,J,K,L)
     \tkzLabelPoint[above left](A){$-2$} 
     \tkzLabelPoint[above right](B){$-3$}\tkzLabelPoint[above right](H){$-2$}\tkzLabelPoint[above right](I){$-2$}\tkzLabelPoint[above right](C){$-2$}\tkzLabelPoint[above right](G){$-2$}
     \tkzLabelPoint[below left](D){$-2$}\tkzLabelPoint[below left](E){$-2$}\tkzLabelPoint[below left](F){$-2$}\tkzLabelPoint[below left](J){$-2$}\tkzLabelPoint[below left](K){$-2$}\tkzLabelPoint[below left](L){$-2$}
       \end{tikzpicture}\hspace{1cm}
     \begin{tikzpicture}[scale=0.7]
    \tkzDefPoints{0/0/A, 1.5/1/B, 1.5/-1/D, 3/-1/E, 4.5/-1/F, 1.5/0/C, 3/0/H, 4.5/0/I, 10.5/-1/G}  
    \tkzDefPoints{6/-1/J, 7.5/-1/K, 9/-1/L, 12/-1/M, 13.5/-1/N, 15/-1/O} 
    \tkzDrawSegment(A,B)\tkzDrawSegment(A,D)\tkzDrawSegment(E,D)\tkzDrawSegment(B,A)\tkzDrawSegment(A,C)\tkzDrawSegment(C,H)\tkzDrawSegment(I,H)
    \tkzDrawSegment(L,G)\tkzDrawSegment(E,F)\tkzDrawSegment(J,F)\tkzDrawSegment(J,K)\tkzDrawSegment(K,L)\tkzDrawSegment(M,L)\tkzDrawSegment(M,N)\tkzDrawSegment(N,O)
    \tkzDrawPoints[fill,black,size=5](A,B,C,D,E,F,G,H,I,J,K,L,M,N,O)
     \tkzLabelPoint[above left](A){$-2$} 
     \tkzLabelPoint[above left](B){$-3$}\tkzLabelPoint[above right](H){$-2$}\tkzLabelPoint[above right](I){$-2$}\tkzLabelPoint[above right](C){$-2$}
     \tkzLabelPoint[below left](G){$-2$}\tkzLabelPoint[below left](M){$-2$}\tkzLabelPoint[below left](O){$-2$}\tkzLabelPoint[below left](D){$-2$}\tkzLabelPoint[below left](E){$-2$}\tkzLabelPoint[below left](F){$-2$}\tkzLabelPoint[below left](J){$-2$}\tkzLabelPoint[below left](K){$-2$}\tkzLabelPoint[below left](N){$-2$}\tkzLabelPoint[below left](L){$-2$}
       \end{tikzpicture} 
     \caption{\smaller[1]{The standard graph of $\Sigma(2,7,11)$ (left), $\Sigma(3,5,7)$ (right) and $\Sigma(3,4,11)$ (bottom). All of these Brieskorn spheres have positive correction term. Note that by reversing the orientation we obtain the Seifert fibred spaces $M(-1;\frac{4}{11},\frac{1}{2},\frac{1}{7}),$ $M(-1;\frac{2}{3},\frac{1}{5},\frac{1}{7})$ and $M(-1;\frac{2}{3},\frac{1}{4},\frac{1}{11})$. }}
     \label{Family1}
\end{figure} 

\subsection{The case when \texorpdfstring{$\mathbf{d_3}$}{d3} is non-vanishing}
We use the same argument of the previous subsection. We need to find all the remaining canonically oriented Brieskorn spheres $Y$ with exactly two symplectically fillable structures and $e_0\leq-2$. We keep the same notation as above; in particular, when $e_0=-3$ we have that $-Y=M(-1;\frac{1}{a},\frac{1}{b},\frac{1}{c},\frac{1}{d})$ for $a,b,c,d\geq2$ such that $\frac{1}{a}+\frac{1}{b}+\frac{1}{c}+\frac{1}{d}=1+\frac{1}{abcd}$. Proposition \ref{prop:number_theory2} then finds us two Brieskorn spheres of this kind; namely, the manifolds $\Sigma(2,3,7,41)$ and $\Sigma(2,3,11,13)$, whose standard graphs have four legs with vertices with framing $-2$.

The $d_3$-invariant is computed easily from Equation \eqref{eq:d33}, while for the correction term we use the following characteristic vector \[V=\begin{pmatrix} -1 & 2 & 0 & \cdots \end{pmatrix}\in H^2(X_Y;\Z)\] where $Y$ is either $\Sigma(2,3,7,41)$ or $\Sigma(2,3,11,13)$ and $V$ ends with $48$ and $24$ zeros respectively. The vector $V$ ends correctly and realizes the maximal Maslov grading, where $X$ is the negative-definite Stein domain given by $\Gamma$. By \cite[Theorem 1.2]{OSz-fullpath} we obtain that $M(V)$ is equal to $12$ and $6$ respectively.

When $e_0=-2$ we recall that $Y$ has standard graph $\Gamma$ as in Figure \ref{Family2}, and that $s$ is the number of circles with framing $-2$ before the $-3$ on the first leg of $\Gamma$. We distinguish two cases:
\begin{itemize}[leftmargin=0.5cm]
    \item $s=0$ \\
      We find five Brieskorn spheres of this kind, see Figure \ref{Family4}; the number theory part of the proof is again postponed to Proposition \ref{prop:number_theory}. As in the previous subsection, the full path algorithm gives us the characteristic vectors \[V_1=\setcounter{MaxMatrixCols}{30}\begin{pmatrix} 0 & -1 & 2 & 0 & 0 & 0 & 0 & 0 & 0 & 0 & 0 & 0 & 0 & 0 & 0\end{pmatrix}\in H^2(X_{\Sigma(3,7,10)};\Z)\:,\] \[V_2=\setcounter{MaxMatrixCols}{30}\begin{pmatrix} 0 & -1 & 0 & 0 & 0 & 0 & 0 & 0 & 0 & 0 & 0 & 2\end{pmatrix}\in H^2(X_{\Sigma(4,5,9)};\Z)\:,\] \[V_3=\setcounter{MaxMatrixCols}{30}\begin{pmatrix} 0 & -1 & 0 & 2 & 0 & 0 & 0 & 0 & 0 & 0 & 0 & 0 & 0 & 0 & 0\end{pmatrix}\in H^2(X_{\Sigma(3,8,11)};\Z)\:,\] \[V_4=\setcounter{MaxMatrixCols}{30}\begin{pmatrix} 0 & -1 & 0 & 0 & 0 & 2 & 0 & 0 & 0 & 0 & 0 & 0 & 0 & 0 & 0 & 0 & 0 & 0\end{pmatrix}\in H^2(X_{\Sigma(3,7,19)};\Z)\] and \[V_5=\setcounter{MaxMatrixCols}{30}\begin{pmatrix} 0 & -1 & 2 & 0 & 0 & 0 & 0 & 0 & 0 & 0 & 0 & 0 & 0 & 0 & 0 & 0 & 0 & 0\end{pmatrix}\in H^2(X_{\Sigma(3,5,14)};\Z)\] which end correctly and realize the maximal Maslov grading. By \cite[Theorem 1.2]{OSz-fullpath} we obtain that the corresponding correction terms are equal to $M(V_i)=2$ for $i=1,...,4$ and $M(V_5)=4$. The $d_3$-invariant is computed from Equation \eqref{eq:d3}.   \\ 
    
 \item $s>0$ \\
  Proposition \ref{prop:number_theory} this time finds us infinitely many Brieskorn spheres of this kind: the family $\Sigma(2,3,6k-1)$ with $k\geq2$ and $\Sigma(2,5,9)$, see Figure \ref{Family5}. The vectors of the full path that gives the correction term are \[V_1=\setcounter{MaxMatrixCols}{30}\begin{pmatrix} 0 & 0 & -1 & 0 & 0 & 0 & 0 & 0 & 0 & 0 & 0 & 0\end{pmatrix}\in H^2(X_{\Sigma(2,5,9)};\Z)\] and \[V_{2,k}=\setcounter{MaxMatrixCols}{30}\begin{pmatrix} 0 & 0 & 0 & 0 & 0 & -1 & 0 & \cdots & 0\end{pmatrix}\in H^2(X_{\Sigma(2,3,6k-1)};\Z)\] providing $d(Y)=M(V_1)=M(V_{2,k})=2$ for $k\geq2$.
\end{itemize}

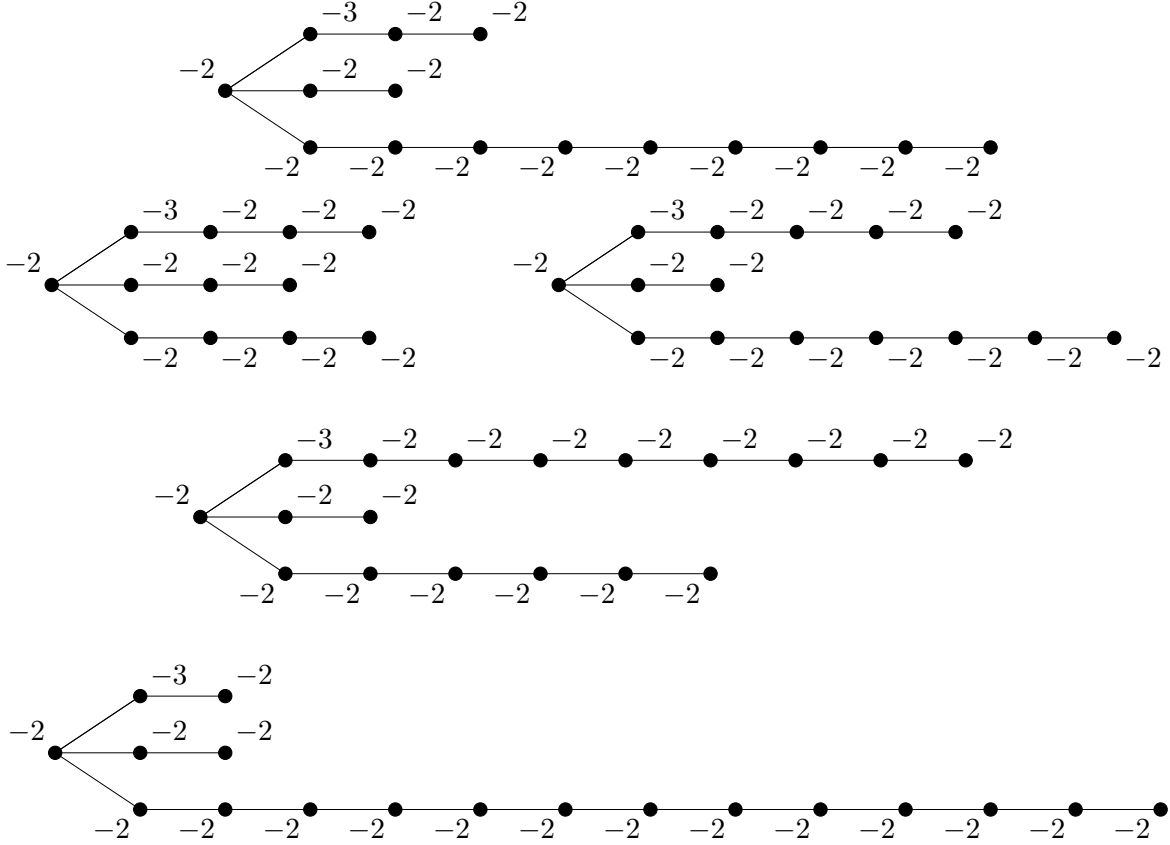
\begin{figure}[t] 
 \begin{tikzpicture}[scale=0.75]
    \tkzDefPoints{0/0/A, 1.5/1/B, 1.5/-1/D, 3/-1/E, 4.5/-1/F, 1.5/0/C, 3/0/H, 10.5/-1/I, 12/-1/G, 3/1/N, 4.5/1/O}  
    \tkzDefPoints{6/-1/J, 7.5/-1/K, 9/-1/L, 13.5/-1/M} 
    \tkzDrawSegment(A,B)\tkzDrawSegment(A,D)\tkzDrawSegment(E,D)\tkzDrawSegment(B,A)\tkzDrawSegment(A,C)\tkzDrawSegment(C,H)\tkzDrawSegment(I,L)
    \tkzDrawSegment(I,G)\tkzDrawSegment(E,F)\tkzDrawSegment(J,F)\tkzDrawSegment(J,K)\tkzDrawSegment(K,L)\tkzDrawSegment(M,L)\tkzDrawSegment(B,N)\tkzDrawSegment(N,O)
    \tkzDrawPoints[fill,black,size=5](A,B,C,D,E,F,G,H,I,J,K,L,M,N,O)
     \tkzLabelPoint[above left](A){$-2$} 
     \tkzLabelPoint[above right](B){$-3$}\tkzLabelPoint[above right](H){$-2$}\tkzLabelPoint[below left](I){$-2$}\tkzLabelPoint[above right](C){$-2$}\tkzLabelPoint[below left](G){$-2$}
     \tkzLabelPoint[below left](D){$-2$}\tkzLabelPoint[below left](E){$-2$}\tkzLabelPoint[below left](F){$-2$}\tkzLabelPoint[below left](J){$-2$}\tkzLabelPoint[below left](K){$-2$}\tkzLabelPoint[below left](L){$-2$}\tkzLabelPoint[below left](M){$-2$}\tkzLabelPoint[above right](N){$-2$}\tkzLabelPoint[above right](O){$-2$}
       \end{tikzpicture}\hspace{1cm}
 \begin{tikzpicture}[scale=0.7]
    \tkzDefPoints{0/0/A, 1.5/1/B, 1.5/-1/D, 3/-1/E, 4.5/-1/F, 1.5/0/C, 3/1/H, 4.5/1/I, 6/1/G}  
    \tkzDefPoints{6/-1/J, 3/0/K, 4.5/0/L} 
    \tkzDrawSegment(A,B)\tkzDrawSegment(A,D)\tkzDrawSegment(E,D)\tkzDrawSegment(B,A)\tkzDrawSegment(A,C)\tkzDrawSegment(B,H)\tkzDrawSegment(I,H)
    \tkzDrawSegment(I,G)\tkzDrawSegment(E,F)\tkzDrawSegment(J,F)\tkzDrawSegment(C,K)\tkzDrawSegment(K,L)
    \tkzDrawPoints[fill,black,size=5](A,B,C,D,E,F,G,H,I,J,K,L)
     \tkzLabelPoint[above left](A){$-2$} 
     \tkzLabelPoint[above right](B){$-3$}\tkzLabelPoint[above right](H){$-2$}\tkzLabelPoint[above right](I){$-2$}\tkzLabelPoint[above right](C){$-2$}\tkzLabelPoint[above right](G){$-2$}
     \tkzLabelPoint[below right](D){$-2$}\tkzLabelPoint[below right](E){$-2$}\tkzLabelPoint[below right](F){$-2$}\tkzLabelPoint[below right](J){$-2$}\tkzLabelPoint[above right](K){$-2$}\tkzLabelPoint[above right](L){$-2$}
       \end{tikzpicture}\hspace{1cm}\vspace{0.5cm}
     \begin{tikzpicture}[scale=0.7]
    \tkzDefPoints{0/0/A, 1.5/1/B, 1.5/-1/D, 3/-1/E, 4.5/-1/F, 1.5/0/C, 3/1/H, 4.5/1/I, 6/1/G, 7.5/1/M, 3/0/N, 10.5/-1/O}  
    \tkzDefPoints{6/-1/J, 7.5/-1/K, 9/-1/L} 
    \tkzDrawSegment(A,B)\tkzDrawSegment(A,D)\tkzDrawSegment(E,D)\tkzDrawSegment(B,A)\tkzDrawSegment(A,C)\tkzDrawSegment(B,H)\tkzDrawSegment(I,H)
    \tkzDrawSegment(I,G)\tkzDrawSegment(E,F)\tkzDrawSegment(J,F)\tkzDrawSegment(J,K)\tkzDrawSegment(K,L)\tkzDrawSegment(G,M)\tkzDrawSegment(C,N)\tkzDrawSegment(O,L)
    \tkzDrawPoints[fill,black,size=5](A,B,C,D,E,F,G,H,I,J,K,L,M,N,O)
     \tkzLabelPoint[above left](A){$-2$} 
     \tkzLabelPoint[above right](B){$-3$}\tkzLabelPoint[above right](H){$-2$}\tkzLabelPoint[above right](I){$-2$}\tkzLabelPoint[above right](C){$-2$}\tkzLabelPoint[above right](G){$-2$}
     \tkzLabelPoint[below right](D){$-2$}\tkzLabelPoint[below right](E){$-2$}\tkzLabelPoint[below right](F){$-2$}\tkzLabelPoint[below right](J){$-2$}\tkzLabelPoint[below right](K){$-2$}\tkzLabelPoint[below right](L){$-2$}\tkzLabelPoint[above right](M){$-2$}\tkzLabelPoint[above right](N){$-2$}\tkzLabelPoint[below right](O){$-2$}
       \end{tikzpicture}\hspace{1cm}\vspace{0.5cm}
       \begin{tikzpicture}[scale=0.75]
    \tkzDefPoints{0/0/A, 1.5/1/B, 1.5/-1/D, 3/-1/E, 4.5/-1/F, 1.5/0/C, 3/0/H, 3/1/I, 4.5/1/G, 6/1/M, 7.5/1/N, 9/1/O, 10.5/1/P, 12/1/Q, 13.5/1/R}  
    \tkzDefPoints{6/-1/J, 7.5/-1/K, 9/-1/L} 
    \tkzDrawSegment(A,B)\tkzDrawSegment(A,D)\tkzDrawSegment(E,D)\tkzDrawSegment(B,A)\tkzDrawSegment(A,C)\tkzDrawSegment(C,H)\tkzDrawSegment(I,G)
    \tkzDrawSegment(I,B)\tkzDrawSegment(E,F)\tkzDrawSegment(J,F)\tkzDrawSegment(J,K)\tkzDrawSegment(K,L)\tkzDrawSegment(G,M)\tkzDrawSegment(M,N)\tkzDrawSegment(N,O)\tkzDrawSegment(O,P)\tkzDrawSegment(P,Q)\tkzDrawSegment(Q,R)
    \tkzDrawPoints[fill,black,size=5](A,B,C,D,E,F,G,H,I,J,K,L,M,N,O,P,Q,R)
     \tkzLabelPoint[above left](A){$-2$} 
     \tkzLabelPoint[above right](B){$-3$}\tkzLabelPoint[above right](H){$-2$}\tkzLabelPoint[above right](I){$-2$}\tkzLabelPoint[above right](C){$-2$}\tkzLabelPoint[above right](G){$-2$}
     \tkzLabelPoint[below left](D){$-2$}\tkzLabelPoint[below left](E){$-2$}\tkzLabelPoint[below left](F){$-2$}\tkzLabelPoint[below left](J){$-2$}\tkzLabelPoint[below left](K){$-2$}\tkzLabelPoint[below left](L){$-2$}\tkzLabelPoint[above right](M){$-2$}\tkzLabelPoint[above right](N){$-2$}\tkzLabelPoint[above right](O){$-2$}\tkzLabelPoint[above right](P){$-2$}\tkzLabelPoint[above right](Q){$-2$}\tkzLabelPoint[above right](R){$-2$}
       \end{tikzpicture}\hspace{1cm}
     \begin{tikzpicture}[scale=0.75]
    \tkzDefPoints{0/0/A, 1.5/1/B, 1.5/-1/D, 3/-1/E, 4.5/-1/F, 1.5/0/C, 3/0/H, 3/1/I, 10.5/-1/G, 16.5/-1/P, 18/-1/Q, 19.5/-1/R}  
    \tkzDefPoints{6/-1/J, 7.5/-1/K, 9/-1/L, 12/-1/M, 13.5/-1/N, 15/-1/O} 
    \tkzDrawSegment(A,B)\tkzDrawSegment(A,D)\tkzDrawSegment(E,D)\tkzDrawSegment(B,A)\tkzDrawSegment(A,C)\tkzDrawSegment(C,H)\tkzDrawSegment(I,B)
    \tkzDrawSegment(L,G)\tkzDrawSegment(E,F)\tkzDrawSegment(J,F)\tkzDrawSegment(J,K)\tkzDrawSegment(K,L)\tkzDrawSegment(M,L)\tkzDrawSegment(M,N)\tkzDrawSegment(N,O)\tkzDrawSegment(P,O)\tkzDrawSegment(P,Q)\tkzDrawSegment(Q,R)
    \tkzDrawPoints[fill,black,size=5](A,B,C,D,E,F,G,H,I,J,K,L,M,N,O,P,Q,R)
     \tkzLabelPoint[above left](A){$-2$} 
     \tkzLabelPoint[above right](B){$-3$}\tkzLabelPoint[above right](H){$-2$}\tkzLabelPoint[above right](I){$-2$}\tkzLabelPoint[above right](C){$-2$}
     \tkzLabelPoint[below left](G){$-2$}\tkzLabelPoint[below left](M){$-2$}\tkzLabelPoint[below left](O){$-2$}\tkzLabelPoint[below left](D){$-2$}\tkzLabelPoint[below left](E){$-2$}\tkzLabelPoint[below left](F){$-2$}\tkzLabelPoint[below left](J){$-2$}\tkzLabelPoint[below left](K){$-2$}\tkzLabelPoint[below left](N){$-2$}\tkzLabelPoint[below left](L){$-2$}\tkzLabelPoint[below left](P){$-2$}\tkzLabelPoint[below left](Q){$-2$}\tkzLabelPoint[below left](R){$-2$}
       \end{tikzpicture} 
     \caption{\smaller[1]{The standard graph of $\Sigma(3,7,10)$ (top), $\Sigma(4,5,9)$ (left), $\Sigma(3,8,11)$ (right), $\Sigma(3,7,19)$ (middle) and $\Sigma(3,5,14)$ (bottom). Note that by reversing the orientation we obtain the Seifert fibred spaces $M(-1;\frac{4}{7},\frac{1}{3},\frac{1}{10})$, $M(-1;\frac{5}{9},\frac{1}{4},\frac{1}{5}),$ $M(-1;\frac{6}{11},\frac{1}{3},\frac{1}{8})$, $M(-1;\frac{10}{19},\frac{1}{3},\frac{1}{7})$ and $M(-1;\frac{3}{5},\frac{1}{3},\frac{1}{14})$.}}
     \label{Family4}
\end{figure}

\medskip

We can now prove Theorem \ref{teo:two}.

\begin{proof}[Proof of Theorem \ref{teo:two}]
 Suppose that $Y$ is a canonically oriented Brieskorn sphere. When the $d_3$-invariant is vanishing the claim follows from Subsection \ref{subsection:first} and \cite[Proposition 2.2]{ACM}, while when $d_3$ is non-vanishing from the three cases discussed above. 
 
 From convex surface theory we have that $\xi_\text{can}$ and $\overline\xi_\text{can}$ are then the only tight structures on each small Brieskorn sphere in Table \ref{Table}. Since $\Sigma(2,3,7,41)$ and $\Sigma(2,3,11,13)$ have a $\Gamma$ with $4$ legs, we can take any negative-twisting structure on them, and apply Giroux torsion along the incompressible torus; the resulting structure is tight with vanishing contact invariant, and then non-fillable because of standard Heegaard Floer results.

 We now consider an oppositely oriented Brieskorn sphere $-Y$. We proved in \cite[Corollary 1.8 and Section 8]{CM} that if $-Y$ is not $-\Sigma(2,3,6k\pm1)$ for any $k\geq1$ then it always carries a symplectically but not Stein fillable structure which is tangent to the fibers, and combining this fact with \cite[Remark 1.14]{CM} we obtain that $-Y$ necessarily has at least three fillable structures, up to isotopy, in this case. When $-Y=-\Sigma(2,3,6k\pm1)$ for some $k\geq1$ then the tight structures have been classified in \cite{EH,GvHM,Tosun} and their number is $\frac{k(k\pm1)}{2}$, which is never equal to two.
\end{proof}

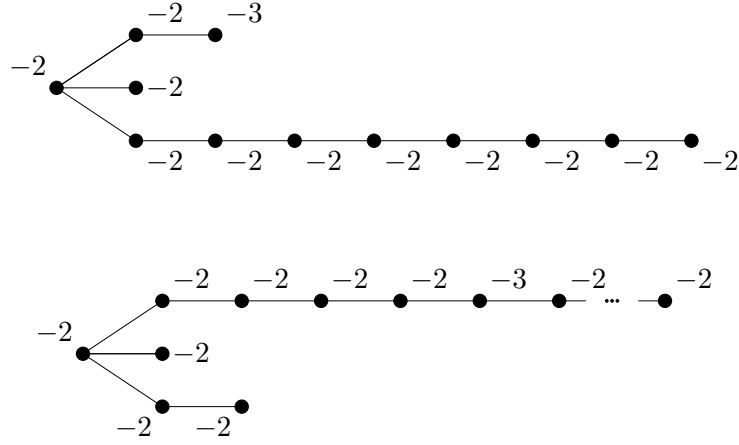
\begin{figure}[t] 
 \begin{tikzpicture}[scale=0.7]
    \tkzDefPoints{0/0/A, 1.5/1/B, 1.5/-1/D, 3/-1/E, 4.5/-1/F, 1.5/0/C, 3/1/H, 10.5/-1/I, 12/-1/G}  
    \tkzDefPoints{6/-1/J, 7.5/-1/K, 9/-1/L} 
    \tkzDrawSegment(A,B)\tkzDrawSegment(A,D)\tkzDrawSegment(E,D)\tkzDrawSegment(B,A)\tkzDrawSegment(A,C)\tkzDrawSegment(B,H)\tkzDrawSegment(I,L)
    \tkzDrawSegment(I,G)\tkzDrawSegment(E,F)\tkzDrawSegment(J,F)\tkzDrawSegment(J,K)\tkzDrawSegment(K,L)
    \tkzDrawPoints[fill,black,size=5](A,B,C,D,E,F,G,H,I,J,K,L)
     \tkzLabelPoint[above left](A){$-2$} 
     \tkzLabelPoint[above right](B){$-2$}\tkzLabelPoint[above right](H){$-3$}\tkzLabelPoint[below right](I){$-2$}\tkzLabelPoint[right](C){$-2$}\tkzLabelPoint[below right](G){$-2$}
     \tkzLabelPoint[below right](D){$-2$}\tkzLabelPoint[below right](E){$-2$}\tkzLabelPoint[below right](F){$-2$}\tkzLabelPoint[below right](J){$-2$}\tkzLabelPoint[below right](K){$-2$}\tkzLabelPoint[below right](L){$-2$}
       \end{tikzpicture}\vspace{1cm} 
     \begin{tikzpicture}[scale=0.7]
    \tkzDefPoints{0/0/A, 1.5/0/B, 1.5/1/D, 3/1/E, 4.5/1/F, 1.5/-1/C, 3/-1/N} 
    \tkzDefPoints{6/1/J, 7.5/1/K, 9/1/L, 11/1/Q} 
    \tkzDefPoints{9.5/1/O}\tkzDefPoints{9.9/1/X, 10/1/Y, 10.1/1/Z}\tkzDefPoints{10.5/1/P}
    \tkzDrawPoints[fill,black,size=1](X,Y,Z)
    \tkzDrawSegment(A,B)\tkzDrawSegment(A,D)\tkzDrawSegment(E,D)\tkzDrawSegment(B,A)\tkzDrawSegment(A,C)
    \tkzDrawSegment(E,F)\tkzDrawSegment(J,F)\tkzDrawSegment(J,K)\tkzDrawSegment(K,L)\tkzDrawSegment(C,N)\tkzDrawSegment(O,L)\tkzDrawSegment(P,Q)
    \tkzDrawPoints[fill,black,size=5](A,B,C,D,E,F,J,K,L,N,Q)
     \tkzLabelPoint[above left](A){$-2$} 
     \tkzLabelPoint[right](B){$-2$}\tkzLabelPoint[below left](C){$-2$}
     \tkzLabelPoint[above right](D){$-2$}\tkzLabelPoint[above right](E){$-2$}\tkzLabelPoint[above right](F){$-2$}\tkzLabelPoint[above right](J){$-2$}\tkzLabelPoint[above right](K){$-3$}\tkzLabelPoint[above right](L){$-2$}\tkzLabelPoint[below left](N){$-2$}\tkzLabelPoint[above right](Q){$-2$}
       \end{tikzpicture}
 \caption{\smaller[1]{The standard graph of $\Sigma(2,5,9)$ (top) and $\Sigma(2,3,6k-1)$ with $k\geq2$ (bottom). There are $k-2$ vertices with framing $-2$ after the one with $-3$ on the first leg of the graph of $\Sigma(2,3,6k-1)$. Note that by reversing the orientation we obtain the Seifert fibered spaces $M(-1;\frac{2}{5},\frac{1}{2},\frac{1}{9})$ and $M(-1;\frac{k}{6k-1},\frac{1}{2},\frac{1}{3})$.}}
     \label{Family5}
\end{figure}

We now show that the Brieskorn spheres in Theorem \ref{teo:two}, together with $\Sigma(2,3,5),$ $-\Sigma(2,3,7)$ and $-\Sigma(2,3,11)$, are the only ones with a unique symplectically fillable contact structure up to contactomorphism.

\begin{proof}[Proof of Proposition \ref{prop:contactomorphism}]
  It follows from \cite[Theorem 7.6]{DHM} that every non-trivial diffeomorphism of a Brieskorn sphere is isotopic to the covering involution, and the latter one acts as conjugation on contact structures. Therefore, for every Brieskorn sphere if two fillable structures are contactomorphic but not isotopic then they are the conjugate one of the other by \cite[Theorem 1.1]{CM-negative}; the claim then follows from Theorem \ref{teo:two}. The fact that $\Sigma(2,3,5)$ carries a unique tight structure is well-known, and follows from the techniques in \cite{EH}, while that the same holds for $-\Sigma(2,3,7)$ and $-\Sigma(2,3,11)$ from \cite{GvHM,Tosun}.
\end{proof}

\section{Some details of number theory} 
We finish the proof that the Brieskorn spheres in Figures \ref{Family1}, \ref{Family4} and \ref{Family5} are the only ones which admit exactly two symplectically fillable contact structures and have $e_0\leq-2$. This relies on the following elementary number theory facts\footnote{The results in this section have been obtained with the aid of OpenAI’s ChatGPT, GPT-5.2.}.

\begin{prop}
 \label{prop:number_theory2}   
 Suppose that the integral quadruple $(a,b,c,d)$ with $d>c>b>a\geq2$ satisfies the following properties:
 \begin{enumerate}
     \item $a,b,c,d$ are pairwise coprime;
     \item $\dfrac{1}{a}+\dfrac{1}{b}+\dfrac{1}{c}+\dfrac{1}{d}=1+\dfrac{1}{abcd}$.
 \end{enumerate}
 Then the only possibilities are $(2,3,7,41)$ and $(2,3,11,13)$.
\end{prop}
\begin{proof}
 If $a>2$ then by Property 1) one has \[\dfrac{1}{a}+\dfrac{1}{b}+\dfrac{1}{c}+\dfrac{1}{d}\leq\dfrac{1}{3}+\dfrac{1}{4}+\dfrac{1}{5}+\dfrac{1}{7}<1\:;\] hence, we have that $a=2$ from Property 2). In the same way if $b>3$ then \[\dfrac{1}{2}+\dfrac{1}{b}+\dfrac{1}{c}+\dfrac{1}{d}\leq\dfrac{1}{2}+\dfrac{1}{5}+\dfrac{1}{7}+\dfrac{1}{9}<1\] which means $b=3$.   
 Plugging these values in Property 2) we obtain \[(c-6)(d-6)=(cd-6c-6d)+36=-1+36=35\] whose only positive solutions of the form $(c,d)$ are $(7,41)$ and $(11,13)$.
\end{proof}

\begin{prop}
 \label{prop:new}   
 There is no integral quadruple $(a,b,c,d)$ with $c>b>a\geq2$ and $d\geq3$ satisfying \[\dfrac{1}{a}+\dfrac{1}{b}+\dfrac{1}{c}+\dfrac{d+1}{2d}=2+\dfrac{1}{abcd}\:.\]
\end{prop}
\begin{proof}
 We rearrange the equation as $d(3abc-2ab-2ac-2bc)=abc-2$. Since $d\geq3$ we have $8abc+2\leq6ab+6ac+6bc$, which after dividing by $2abc$ yields $\frac{1}{a}+\frac{1}{b}+\frac{1}{c}>\frac{4}{3}$. Clearly, we cannot have $a\geq3$, thus $a=2$ and then $2d(bc-b-c)=bc-1$. We now show that this implies $d\leq2$ which is a contradiction: we have that $d<3$ is equivalent to $5bc-6b-6c+1>0$, this is obviously satisfied for $c>b\geq3$ because \[5bc-6b-6c+1>15c-12c+1>10\:.\]
\end{proof}

\begin{prop}
 \label{prop:number_theory}
 Suppose that the integral triple $(a,b,uv-1)$  with $b>a\geq2$ and $u,v\geq2$ satisfies the following properties:
 \begin{enumerate}
     \item $a,b$ and $uv-1$ are pairwise coprime;
     \item $\dfrac{1}{a}+\dfrac{1}{b}+\dfrac{v}{uv-1}>1$;
     \item $vab-1=(uv-1)X$ where $X:=ab-a-b$;
     \item $a+b+u+v=4+(v-1)[ab-(u-1)X]$.
 \end{enumerate}
 Then the only possibilities are $(5,7,3),$ $(4,11,3)$ and $(2,7,11)$. Furthermore, if we drop Property 4) then we also find $(3,14,5),$ $(3,10,7),$ $(4,5,9),$ $(3,8,11),$ $(3,7,19),$ $(2,9,5)$ and $(2,3,6v-1)$. 
\end{prop}
We split the proof into a couple of lemmas. Throughout the section, we always assume that $a,b,u,v$ are as in the statement of Proposition \ref{prop:number_theory}; moreover, we call a triple satisfying Properties 1) -- 3) an \emph{admissible triple}.
\begin{lemma}
 \label{lemma1}
 If a triple $(a,b,uv-1)$ is admissible then $u\in\{2,3,6\}$; in particular, when $u=6$ the only possibility is $(2,3,6v-1)$. 
\end{lemma}
\begin{proof}
 We start by showing that $2\leq u\leq6$. Suppose that $a\geq3$, so that $\frac{1}{a}+\frac{1}{b}\leq\frac{2}{3}$; from Property 2) we get $\frac{v}{uv-1}>\frac{1}{3}$, which implies $v(u-3)<1$: because $v\geq2$ this forces $u\leq3$. So if $u>3$ then we must have $a=2$. When $a=2$ Property 3) gives the explicit formula
 \begin{equation}
  b=\dfrac{2uv-3}{v(u-2)-1}\:;   
  \label{eq:spades}
 \end{equation}
 an easy manipulation of this identity shows that the condition $b<3$ is equivalent to $v(u-6)>0$, thus if $u>6$ then $2\leq b<3$ which means $b=2$. This contradicts Property 1) because $a$ and $b$ should be coprime.
 
 When $u=4,5,6$ we saw that $a=2$, thus $b\geq2$ is given by Equation \eqref{eq:spades} and if $u<6$ then we find no integral solution; on the other hand, when $u=6$, Equation \eqref{eq:spades} provides $b=3$. 
\end{proof}
We then only need to consider the cases when $u=2,3$. We do this in separate lemmas, even though the argument is the same for both. First, we prove that for each $u$ we only have to consider a finite number of cases.
\begin{lemma}
 \label{lemma4}
 Suppose that a triple $(a,b,uv-1)$ is admissible. We have that if $u=2$ then $2\leq a\leq5$, while if $u=3$ then $a=2,3$. 
\end{lemma}

\begin{proof}
 From Property 2) we have that \[\dfrac1a+\dfrac1b+\dfrac{v}{uv-1}>1\hspace{1cm}\text{ implies }\hspace{1cm}\frac1a+\frac1b>1-\frac{v}{uv-1}=\frac{(u-1)v-1}{uv-1}\:.\]
 Since $a<b$, we obtain $\frac{1}{a}+\frac{1}{b}<\frac{2}{a}$ and then $\frac{2}{a}>\frac{(u-1)v-1}{uv-1}$. If $u=2$ we write 
 \[\dfrac{2}{a}>\dfrac{v-1}{2v-1}\hspace{1cm}\text{ which means }\hspace{1cm}a<\dfrac{4v-2}{v-1}<6\:,\] thus $2\leq a\leq 5$. In addition, if $u=3$ we write 
 \[\dfrac{2}{a}>\dfrac{2v-1}{3v-1}\hspace{1cm}\text{ which means }\hspace{1cm}a<\dfrac{2(3v-1)}{2v-1}<\dfrac{3v-1}{v-1}<4\:,\] thus $2\leq a\leq 3$.
\end{proof}
An easy manipulation of Property 3) yields 
\begin{equation}\begin{aligned}
  u=2:\:\hspace{1cm}&\bigl[(v-1)a-(2v-1)\bigr]\cdot\bigl[(v-1)b-(2v-1)\bigr]=4v^2-5v+2\:;\\[1mm]
  u=3:\:\hspace{1cm}&\bigl[(2v-1)a-(3v-1)\bigr]\cdot\bigl[(2v-1)b-(3v-1)\bigr]=9v^2-8v+2\:.
  \label{eq:bidots}
 \end{aligned}
\end{equation}
For each value of $a$ and $u$ in Lemmas \ref{lemma1} and \ref{lemma4} we set $t(v)$ equal to the first factor of the left-side members of Equation \eqref{eq:bidots}, and $N(v)$ to the right-side members of Equation \eqref{eq:bidots}.

\begin{lemma}
 \label{lemma2}
 If $u=2$ then the admissible triples are exactly the ones listed below.
\end{lemma}

\[
\begin{array}{c|c|c|c|l|l}
a & a^2-a+2 & \text{Divisors }t & v=\dfrac{t+a-1}{a-2}\ (\ge2) & b & \text{Valid triples}\\
\hline
2 & 4 & - & - & - & \text{none} \\
3 & 8 & 1,2,4,8 & 3,4,6,10 & 14,10,8,7 & (3,14,5),(3,10,7), \\
 & & & & & (3,8,11),(3,7,19)\\
4 & 14 & 1,2,7,14 & 2,5 & 11,5 & (4,11,3),(4,5,9)\\
5 & 22 & 1,2,11,22 & 2,5 & 7 & (5,7,3)
\end{array}\label{Table2}
\]
\begin{proof}
 Let $t(v)=(v-1)a-(2v-1)=v(a-2)-(a-1)$ and $N(v)=4v^2-5v+2$. For $a>2$, substituting $v=\frac{t(v)+(a-1)}{a-2}$ and simplifying denominators gives \[(a-2)^2 N(v)=4t(v)^2+(3a+2)t(v)+(a^2-a+2)\:;\] hence, Equation \eqref{eq:bidots} implies that $t(v)$ is a (positive) divisor of $N(v)$, so $t(v)\mid a^2-a+2$. For $a=2$ we see that Equation \eqref{eq:bidots} is not satisfied by any $b$ positive.

 For each value of $a$ and each $t$ positive divisor of $a^2-a+2$ then $v=\frac{t+a-1}{a-2}$ must be an integer at least equal to 2. Each such pair $(a,v)$ gives at most one $b$ from Equation \eqref{eq:bidots}.
\end{proof}

\begin{lemma}
 \label{lemma3}
 If $u=3$ then the admissible triples are exactly the ones listed below.
\end{lemma}

\[
\begin{array}{c|c|c|c|l|l}
a & a^2-2a+3 & \text{Divisors }t & v=\dfrac{t+a-1}{2a-3}\ (\ge2) & b & \text{Valid triples}\\
\hline
2 & 3 & 1,3 & 2,4 & 9,7 & (2,9,5),(2,7,11)\\
3 & 6 & 1,2,3,6 & \text{no integer }\ge2 & - & \text{none}\label{Table3}
\end{array}
\]

\begin{proof}
 Let $t(v)=(2v-1)a-(3v-1)=v(2a-3)-(a-1)$ and $N(v)=9v^2-8v+2$. Substituting $v=\frac{t(v)+(a-1)}{2a-3}$ and simplifying denominators again yields \[(2a-3)^2 N(v)=9t(v)^2+(2a+6)t(v)+(a^2-2a+3)\:;\] so that $t(v)$ is a positive divisor of $a^2-2a+3$ as in Lemma \ref{lemma2}. Let $t(v)=t$ for any such divisor, then $v=\frac{t+a-1}{2a-3}$ must be an integer at least equal to 2, and we conclude as in the previous case.
\end{proof}
We can now finish the proof.
\begin{proof}[Proof of Proposition \ref{prop:number_theory}]
 By Lemma \ref{lemma1} we restrict the values of $u$ to $2,3,6$ and we know how to treat the case $u=6$. By Lemma \ref{lemma4} we reduce the cases $u=2$ and $u=3$ to a finite check, and the above tables in Lemmas \ref{lemma2} and \ref{lemma3} give the candidate admissible triples. It is easy to see that all the listed triples satisfy Properties 1) -- 3).
\end{proof}


\begin{thebibliography}{9}
 \bibitem{ACM} A. Alfieri, A. Cavallo and I. Matkovi\v c, \emph{Brieskorn spheres and rational homology ball symplectic fillings},
   arXiv:2605.13812. 
    \bibitem{CM-negative} A. Cavallo and I. Matkovi\v c, \emph{Fillable structures on negative-definite Seifert fibred spaces},
     arXiv:2604.28174. 
    \bibitem{CM} A. Cavallo and I. Matkovi\v c, \emph{Heegaard Floer homology and maximal twisting numbers},
      arXiv:2604.28162.
    \bibitem{DHM} I. Dai, M. Hedden and A. Mallick, \emph{Corks, involutions, and Heegaard Floer homology},
    J. Eur. Math. Soc., \textbf{25} (2023), no. 6, pp. 2319--2389.
   \bibitem{EH} J. Etnyre and K. Honda, \emph{On the nonexistence of tight contact structures},
    Ann. Math. (2), \textbf{153} (2001), no. 3, pp. 749--766. 
   \bibitem{FS} R. Fintushel and R. Stern, \emph{Instanton homology of Seifert fibred homology three spheres}, 
    Proc. London Math. Soc., \textbf{61} (1990), pp. 109--137.
   \bibitem{GvHM} P. Ghiggini and J. Van Horn-Morris, \emph{Tight contact structures on the Brieskorn spheres $-\Sigma(2,3,6n-1)$ and contact invariants},
   J. Reine Angew. Math., \textbf{718} (2016), pp. 1--24.  
  \bibitem{ImC-e} A. Issa and D. McCoy, \emph{On Seifert fibered spaces bounding definite manifolds},
   Pac. J. Math., \textbf{304} (2020), no. 2, pp. 463--480.
  \bibitem{LM} P. Lisca and G. Mati\'c, \emph{Tight contact structures and Seiberg-Witten invariants}, Invent. Math., \textbf{129} (1997), pp. 509--525. 
 \bibitem{OSz-negative} P. Ozsv\'ath and Z. Szab\'o, \emph{Absolutely graded Floer homologies and intersection forms for four-manifolds with boundary},
  Adv. Math., \textbf{173} (2003), pp. 179--261.  
 \bibitem{OSz-fullpath} P. Ozsv\'ath and Z. Szab\'o, \emph{On the Floer homology of plumbed three-manifolds},
  Geom. Topol., \textbf 7 (2003), no. 1, pp. 185--224.
  \bibitem{Saveliev} N. Saveliev, \emph{Invariants for homology $3$-spheres}, 
   Encyclopaedia of Mathematical Sciences 140. Low-Dimensional Topology 1. Berlin: Springer.
 \bibitem{Savk} O. \c{S}avk, \emph{More Brieskorn spheres bounding rational balls},
   Topology Appl., \textbf{286} (2020), 107400, 11 pp.
  \bibitem{Tosun} B. Tosun, \emph{Tight small Seifert fibered manifolds with $e_0=-2$},
   Algebr. Geom. Topol., \textbf{20} (2020), no. 1, pp. 1--27. 
\end{thebibliography}
\end{document}